\documentclass[12pt]{article}
\usepackage[utf8]{inputenc}
\usepackage[T1]{fontenc}
\usepackage{lmodern}
\usepackage[english]{babel}
\usepackage{amsmath,mathtools,amsthm,amssymb,amsfonts}
\usepackage{color}
\usepackage{comment}

\usepackage{mathrsfs}
\usepackage{graphicx}
\usepackage{enumitem}
\usepackage{tikz} 
\usepackage{esint}
\usepackage{pgfplots}
\usepackage{esvect}
\usepackage{pdfpages}
\usepackage{nicefrac}
\usepackage{leftidx, tensor}
\usepackage{cleveref}
\usepackage{accents}
\usetikzlibrary{shapes,arrows}
\usetikzlibrary{calc}
\usetikzlibrary{shapes.geometric}
\usetikzlibrary{shapes.arrows}
\usetikzlibrary{arrows.meta}
\usetikzlibrary{decorations.markings}
\usetikzlibrary{fit}
\usetikzlibrary{patterns}
\usetikzlibrary{hobby}
\usepgfplotslibrary{patchplots}
\pgfplotsset{compat=1.11}

\renewcommand{\thefootnote}{\fnsymbol{footnote}}
\newcommand{\definedas}{\mathrel{\raise.095ex\hbox{\rm :}\mkern-5.2mu=}}

\newcommand{\R}{\mathbb{R}}
\newcommand{\N}{\mathbb{N}}

\newcommand{\Sbb}{\mathbb{S}}

\renewcommand{\d}{\,\mathrm{d}}

\newcommand{\Hess}{\mathrm{Hess}}

\newcommand{\ul}[1]{\underline{#1}}

\newcommand{\btr}[1]{\left\vert#1\right\vert}
\newcommand{\newbtr}[1]{\vert#1\vert}
\newcommand{\norm}[1]{\btr{\btr{#1}}}

\newcommand{\spann}[1]{\left\langle#1\right\rangle}

\newcommand{\Ric}{\mathrm{Ric}}

\newcommand{\Rm}{\mathrm{Rm}}

\newcommand{\two}{\operatorname{II}}
\newcommand{\tr}{\text{tr}}
\newcommand{\dive}{\operatorname{div}}

\theoremstyle{plain}
\newtheorem{thm}{Theorem}[section]
\newtheorem{prop}[thm]{Proposition}

\newtheorem{lem}[thm]{Lemma}

\theoremstyle{definition}

\newtheorem{bem}[thm]{Remark}
\newtheorem{kor}[thm]{Corollary}

\begin{document}
		\begin{center}
			{\LARGE {Ricci flow on surfaces along the standard lightcone in the $3+1$-Minkowski spacetime}\par}
		\end{center}
		\vspace{0.5cm}
		\begin{center}
			{\large Markus Wolff\footnote[2]{wolff@math.uni-tuebingen.de}}\\
			\vspace{0.4cm}
			{\large Department of Mathematics}\\
			{\large Eberhard Karls Universit\"at T\"ubingen, Germany}
		\end{center}
		\vspace{0.4cm}
		\begin{abstract}
			Identifying any conformally round metric on the  $2$-sphere with a unique cross section on the standard lightcone in the $3+1$-Minkowski spacetime, we gain a new perspective on $2d$-Ricci flow on topological spheres. It turns out that in this setting, Ricci flow is equivalent to a null mean curvature flow first studied by Roesch--Scheuer along null hypersurfaces. Exploiting this equivalence, we can translate well-known results from $2d$-Ricci flow first proven by Hamilton into a full classification of the singularity models for null mean curvature flow in the Minkowski lightcone. Conversely, we obtain a new proof of Hamilton's classical result using only the maximum principle.
		\end{abstract}
		\renewcommand{\thefootnote}{\arabic{footnote}}
		\setcounter{footnote}{0}
	\section{Introduction}
	A classical result for $2d$-Ricci flow states that any metric on a $2$-surface converges under renormalized Ricci flow to a metric of constant scalar curvature: This was first proven by Hamilton \cite{hamilton1}. In the case of surfaces of genus $0$ this was initially proven for metrics of strictly positive scalar curvature. This assumption was eventually removed by Chow \cite{chow1}. Their techniques rely heavily on establishing a Harnack inequality and an entropy bound. Their approach also gives an independent proof of the uniformization theorem, c.f. \cite{chenlutian}. Due to the difficulty of this approach, the conformally round case was later revisited and several new proofs were given utilizing the uniformization theorem \cite{andrewsbryan,bartzstruweye,struwe}.
	
	In this paper, we gain a new geometric perspective on $2d$-Ricci flow on topological spheres by uniquely embedding any conformally round metric on the $2$-sphere into the past-pointing standard lightcone in the $3+1$-Minkowski spacetime. This idea of identifying any metric conformal to a a given Riemannian metric as a unique cross section on a null hypersurface within a spacetime satisfying the Einstein Vacuum Equations was first proposed by Fefferman--Graham \cite{feffermangraham} to study conformal invariants. In the context of Lorentzian geometry, Hamilton's initial restriction to metrics of strictly positive scalar curvature translates to a physically reasonable assumption on the representing codimension-$2$ surface in the Minkowski spacetime. From the Gauß equation, we are moreover able to conclude the equivalence of Ricci flow and null mean curvature flow along the lightcone first studied by Roesch--Scheuer \cite{roeschscheuer}. As defined by Roesch--Scheuer, null mean curvature flow 
	\begin{align}
		\frac{\d}{\d t}x=-\frac{1}{2}\spann{\vec{\mathcal{H}},L}\ul{L}
	\end{align}
	describes the evolution proportional to the projection of the mean curvature vector with respect to the null generator $\ul{L}$ of the null hypersurface, where $\vec{\mathcal{H}}$ is the codimension-$2$ mean curvature vector in the ambient spacetime and $L$ is a null vector such that $\{\ul{L},L\}$ forms an appropriate null frame of the normal space of the surface. This projection along the null hypersurface greatly reduces the difficulty of the general codimension-$2$ problem and essentially transforms it into a scalar valued parabolic equation. This was utilized by Roesch--Scheuer in the detection of marginally outer trapped surfaces. No such surfaces exist in the Minkowski spacetime as they arise as the cross section between the null hypersurface and a (black hole) horizon. Due to the equivalence to Ricci flow, we can conclude that the flow extinguishes in finite time along the lightcone and provide a full characterization of the singularity models of the flow (Corollary \ref{thm_singularities}), each corresponding to a member of the restricted Lorentz group $\operatorname{SO}^+(3,1)$. Conversely, understanding conformal Ricci flow as an extrinsic curvature flow along the lightcone gives rise to a new proof of Hamilton's classical result (Theorem \ref{thm_mainthm}). The main ingredient to this approach will be a choice of gauge for the null frame giving rise to a scalar valued second fundamental form along the lightcone. Then any conformally round metric has constant scalar curvature if and only if this scalar valued second fundamental form is pure trace (Proposition \ref{prop_codazziminkowski2}). Studying the evolution of this geometric object along the flow will allow us to adopt techniques utilized in the study of $3d$ Ricci flow and mean curvature flow (cf. \cite{andrewschowguentherlangford,chowluni}).
	
	This observation extends to higher dimensions, as it turns out that null mean curvature flow as defined above is equivalent to the Yamabe flow in the conformal class of the round metric for arbitrary dimension. See Section \ref{sec_discussion} for a more detailed discussion and references. \newline\\
	This paper is structured as follows:\newline
	In \Cref{sec_prelim} we fix some notation and recall the well-known null structure equations in the Minkowski spacetime. In \Cref{sec_nullgeom} we compute all necessary geometric objects on the standard Minkowski lightcone and fix our choice of null gauge. In \Cref{sec_equialence} we establish the equivalence between Ricci flow in the conformal class of the round sphere and null mean curvature flow along the standard Minkowski lightcone, and prove Corollary  \ref{thm_singularities}. In \Cref{sec_hamilton} we prove Theorem \ref{thm_mainthm}, establishing a new proof of Hamilton's classical result. We close with some comments on the Yamabe flow in the higher dimensional case in \Cref{sec_discussion}.
		
		\subsection*{Acknowledgements.}
		I would like to express my sincere gratitude towards my PhD supervisors Carla Cederbaum and Gerhard Huisken for their continuing guidance and helpful discussions.
		\setcounter{section}{1}
	\section{Preliminaries}\label{sec_prelim}
	Throughout this paper, $\R^{3,1}$ will always refer to the $3+1$-dimensional Minkowski spacetime $(\R^{3,1},\eta)$, where $\R^4$ is equipped with the flat metric $\eta$ of signature $(-+++)$. In polar coordinates, $\R^{3,1}$ is given as $\R\times(0,\infty)\times\Sbb^2$ with
	\[
	\eta=-\d t^2+ \d r^2+r^2\d\Omega^2,
	\]
	where $\d\Omega^2$ denotes the standard round metric on $\Sbb^2$. The standard lightcone centered at the origin in the Minkowski spacetime is given as the set
	\[
	C(0):=\{\btr{t}=r\},
	\]
	and we denote the components $C(0)_+:=C(0)\cap\{t\ge 0\}$, $C(0)_{-}:=C(0)\cap\{t\le 0\}$ as the future-pointing and past-pointing standard lightcone (centered at the origin and with time-orientation induced by $\partial_t$), respectively.
	
	In the following, $(\Sigma,\gamma)$ will always denote a $2$-surface with Riemanian metric $\gamma$, and we are in particular interested in such surfaces arising as closed, orientable, spacelike codimension-$2$ surfaces in $\R^{3,1}$ (usually restricted to the lightcone). As usual, we define the vector valued second fundamental form $\vec\two$ of $\Sigma$ in $\R^{3,1}$ as 
	\begin{align*}
	\vec\two(V,W)=\left(\overline{\nabla}_VW\right)^\perp,
	\end{align*}
	where $\overline{\nabla}$ denotes the Levi-Civita connection on the Minkowski spacetime. The codimension-$2$ mean curvature vector of $\Sigma$ is then defined as $\vec{\mathcal{H}}=\tr_\gamma \vec\two$, where $\tr_\gamma$ denotes the metric trace on $\Sigma$ with respect to $\gamma$, and we denote its Lorentzian length by $\mathcal{H}^2:=\eta(\vec{\mathcal{H}},\vec{\mathcal{H}})$. Let $\{\ul{L},L\}$ be a null frame of $\Gamma(T^\perp\Sigma))$, i.e., $\eta(\ul{L},\ul{L})=\eta(L,L)=0$, with $\eta(\ul{L},L)=2$. Then $\vec{\two}$ and $\vec{\mathcal{H}}$ admit the decomposition 
	\begin{align}
	\begin{split}
		\vec{\two}&=-\frac{1}{2}\chi\ul{L}-\frac{1}{2}\ul{\chi}L,\\
		\vec{\mathcal{H}}&=-\frac{1}{2}\theta\ul{L}-\frac{1}{2}\ul{\theta}L.
	\end{split}
	\end{align}
	Here, the null second fundamental forms $\ul{\chi}$ and $\chi$ with respect to $\ul{L}$ and $L$, respectively, are defined as
	\begin{align*}
	\ul{\chi}(V,W)&\definedas \spann{\overline{\nabla}_V\ul{L},W}=-\spann{\overline{\nabla}_VW,\ul{L}},\\
	{\chi}(V,W)&\definedas \spann{\overline{\nabla}_V{L},W}=-\spann{\overline{\nabla}_VW,L},
	\end{align*}
	for all tangent vector fields $V,W\in T\Sigma$, and the null expansions $\ul{\theta}$ and $\theta$ with respect to $\ul{L}$ and $L$, respectively, as
	\begin{align*}
	\ul{\theta}&\definedas \tr_\gamma\ul{\chi},\\
	{\theta}&\definedas \tr_\gamma{\chi}.
	\end{align*}
	In particular
	\begin{align}
	\begin{split}\label{eq_secondffnulldecomp}
	\newbtr{\vec\two}^2&=\spann{\chi,\ul{\chi}},\\
	\mathcal{H}^2&=\ul{\theta}\theta,
	\end{split}
	\end{align}
	and if $\mathcal{H}^2$ is constant along $\Sigma$, i.e., $\vec{\mathcal{H}}$ has constant Lorentzian length along $\Sigma$, we call $\Sigma$ a surface of constant spacetime mean curvature (STCMC), cf. Cederbaum--Sakovich \cite{cederbaumsakovich}.
	Finally, we define the connection one-form $\zeta$ as
	\[
	\zeta(V)\definedas\frac{1}{2}\spann{\overline{\nabla}_V\ul{L},L}
	\]
	and together with $\ul{\chi}$, $\chi$, we collect the following well-known identities:
	\begin{lem}\label{lem_vectorderivatives}
		\begin{align*}
		\overline{\nabla}_{\partial_i}\partial_j&=\nabla_{\partial_i}\partial_j-\frac{1}{2}\ul{\chi}_{ij}L-\frac{1}{2}\chi_{ij}\ul{L},\\
		\overline{\nabla}_{\partial_i}\ul{L}&=\ul{\chi}_i^j\partial_j+\zeta(\partial_i)\ul{L},\\
		\overline{\nabla}_{\partial_i}L&=\chi_i^j\partial_j-\zeta(\partial_i)L.
		\end{align*}
	\end{lem}
	We denote the Riemann curvature tensor, Ricci curvature and scalar curvature on  $(\Sigma,\gamma)$ by $\Rm$, $\Ric$, $\operatorname{R}$, respectively, where we use the following conventions:
	\begin{align*}
	\Rm(X,Y,W,Z)&=\spann{\nabla_X\nabla_YZ-\nabla_Y\nabla_XZ-\nabla_{[X,Y]}Z,W},\\
	\Ric(V,W)&=\tr_\gamma \Rm(V,\cdot,W,\cdot),\\
	\operatorname{R}&=\tr_\gamma\Ric.
	\end{align*}
	We use $\nabla$ for the Levi--Civita connection on $(\Sigma,\gamma)$ and $\nabla^k$ for the $k$-th tensor derivative. We use $\btr{\cdot}:=\btr{\cdot}_{\gamma}$ for all respective tensor norms induced by $\gamma$, where we will usually omit $\gamma$ for convenience when it is clear from the context. By slight abuse of notation, we will denote the gradient of a function $f$ on $(\Sigma,\gamma)$ by $\nabla f$. We denote the trace free part of a $(0,2)$-tensor $T$ as $\accentset{\circ}{T}:=T-\tr_\gamma T\gamma$.
	
	We briefly recall the well-known Gauß and Codazzi Equations in the case of a codimension-$2$ surface in $\R^{3,1}$, which can be directly computed from Lemma \ref{lem_vectorderivatives}. Compare \cite[Theorem 2.2]{wang}\footnote{Note the different sign conventions for $\ul{L}$.}.
	\begin{prop}[Gauß Equations]\label{prop_nullgauß}
		\begin{align*}
		\Rm_{ijkl}&=\frac{1}{2}\chi_{jl}\ul{\chi}_{ik}+\frac{1}{2}\ul{\chi}_{jl}\chi_{ik}-\frac{1}{2}\chi_{jk}\ul{\chi}_{il}-\frac{1}{2}\ul{\chi}_{jk}\chi_{il},\\
		\Ric_{ik}&=\frac{1}{2}\theta\ul{\chi}_{ik}+\frac{1}{2}\ul{\theta}\chi_{ik}-\frac{1}{2}(\chi\cdot\ul{\chi})_{ik}-\frac{1}{2}(\ul{\chi}\cdot\chi)_{ik},\\
		\operatorname{R}&=\mathcal{H}^2-\newbtr{\vec{\two}}^2.
		\end{align*}
	\end{prop}
	\begin{prop}[Codazzi Equations]\label{prop_nullcodazzi}
		\begin{align*}
		\nabla_i\ul{\chi}_{jk}-\nabla_j\ul{\chi}_{ik}
		&=-\zeta_j\ul{\chi}_{ik}+\zeta_i\ul{\chi}_{jk},\\
		\nabla_i\chi_{jk}-\nabla_j\chi_{ik}&=+\zeta_j\chi_{ik}-\zeta_i\chi_{jk}.
		\end{align*}
	\end{prop}	
	In the following, we will always consider null vector fields $\ul{L}$ such that their integral curves are geodesics, i.e.,
	\begin{align}\label{eq_geodesiccondition}
	\overline{\nabla}_{\ul{L}}\ul{L}=0.
	\end{align}
	Under this additional assumption, the propagation equations, also known as the\linebreak Raychaudhuri optical equations (cf. \cite[Lemma 3.2]{roeschscheuer}), along a (local) deformation
	\[
	\frac{\d}{\d t}=\varphi \ul{L}
	\]
	with $\varphi\in C^2(\Sigma)$, simplify in $\R^{3,1}$ to the following:
	\begin{lem}[Propagation Equations]\label{lem_propagation}\,
		\begin{itemize}
			\item[\emph{(i)}] $\frac{\d}{\d t}\gamma_{ij}=2\varphi\ul{\chi}_{ij}$, $\frac{\d}{\d t}\gamma^{ij}=-2\varphi\ul{\chi}^{ij}$, $\frac{\d}{\d t}\d\mu=\varphi\ul{\theta}\d \mu$,
			\item[\emph{(ii)}] 
			$\frac{\d}{\d t}\ul{\chi}_{ij}=\varphi(\ul{\chi})^2_{ij}$,
			\item[\emph{(iii)}]
			$\frac{\d}{\d t}L=-2\nabla\varphi-2\varphi\zeta^k\partial_k$,
			\item[\emph{(iv)}]
			$\frac{\d}{\d t}\chi_{ij}=-2\Hess_{ij}\varphi-2(\d\varphi_i\otimes\zeta_j+\d\varphi_j\otimes\zeta_i)-\varphi\left(2\nabla_i\zeta_j+2\zeta_i\otimes\zeta_j-(\chi\ul{\chi})_{ij}\right)$,
			\item[\emph{(v)}]
			$\frac{\d}{\d t}\ul{\theta}=-\varphi\newbtr{\ul{\chi}}^2$,
			\item[\emph{(vi)}]
			$\frac{\d}{\d t}\theta=-2\Delta\varphi-2\gamma(\nabla\varphi,\zeta)-\varphi\left(\newbtr{\vec{\two}}^2+2\dive\zeta+2\btr{\zeta}^2\right)$.
		\end{itemize}
	\end{lem}

	\section{Null Geometry on the standard Minkowski lightcone}\label{sec_nullgeom}
		For the sake of simplicity, we will keep the discussion of null geometry to the standard Minkowski lightcone. For the interested reader, we refer to \cite{marssoria,roeschscheuer} for a more complete and general introduction to null geometry. For our purpose, it is most convenient to introduce the null coordinates $v\definedas r+t$, $u\definedas r-t$ on $\R^{3,1}$. Then $\eta$ can be written as
		\[
			\eta=\frac{1}{2}\left(\d u\d v+\d v\d u\right)+r^2\d\Omega^2
		\]
		with $r=r(u,v)=\frac{u+v}{2}$. Now, all past-pointing standard lightcones in the Minkowski spacetime are given by the sets $\{v=\operatorname{const.}\}$ (and similarly all future-pointing lightcones are given by $\{u=\operatorname{const.}\}$). From now on, we will work on the null hypersurface $\mathcal{N}=\{v=0\}=C(0)_-$, i.e., the past-pointing standard lightcone centered at the origin, but all identities derived for $\mathcal{N}$ will also analogously hold on all level sets of $v$ and $u$ respectively. Note that $\mathcal{N}$ has the induced degenerate metric 
		\[
			r^2\d\Omega^2,
		\]
		and is generated by the geodesic integral curves of $\ul{L}:=2\partial_u$. Note that $\ul{L}$ is past-pointing and consistent with assumption \eqref{eq_geodesiccondition}. Recall that the null generator $\ul{L}$ of a null hypersurface is both tangential and normal to $\mathcal{N}$, and by choice of $\ul{L}$ we have $\ul{L}(r)=1$. Thus, $r$ restricts to an affine parameter along $\mathcal{N}$. In particular, we can represent any spacelike cross section $\Sigma$ of $\mathcal{N}$ (which intersects any integral curve of $\ul{L}$ exactly once) as a graph over $\Sbb^2$, i.e., $\Sigma=\Sigma_\omega=\{\omega=r\}\subseteq\mathcal{N}$. In particular, $\Sigma$ has the induced metric
		\[
			\gamma=\omega^2\d\Omega^2,
		\]
		so $(\Sigma,\gamma)$ is conformally round. Conversely, for any conformally round metric $\gamma_\omega=\omega^2\d\Omega^2$ there exists a unique spacelike cross section $\Sigma_\omega$ such that $(\Sigma_\omega,\gamma_\omega)$ embeds into $\mathcal{N}$, where we will omit the subscript $\omega$ in the following when it is clear from the context. This observation is similar to an idea developed by Fefferman--Graham \cite{feffermangraham}, and their construction indeed yields the standard lightcone in the $3+1$-Minkowski spacetime in the case of the round $2$-sphere.
		
		We now want to represent the extrinsic curvature of a spacelike cross section $(\Sigma,\gamma)$ of $\mathcal{N}$ as a codimension-$2$ surface with respect to a particular null frame. Recall that the null generator $\ul{L}$ is both tangent and normal to $\mathcal{N}$, in particular $\ul{L}$ is normal to any spacelike cross section $(\Sigma,\gamma)$. We further consider a normal null vector field $L$ along $\Sigma$ such that $\eta(\ul{L},L)=2$. This uniquely determines $L$ such that $\{\ul{L},L\}$ form a frame of the normal bundle $T^\perp\Sigma$ of $\Sigma$. Note that $L$ is future-pointing.
		
		We now remark that the standard round spheres $\{\Sigma_s\}_{s\in(0,\infty)}$ form a level-set foliation with respect to the integral curves of the null generator $\ul{L}$. It is easy to verify that for any leaf $\Sigma_s$, we have $L=2\partial_v$ and find
		\begin{align*}
			\ul{\chi}_s&=\chi_s=s\d\Omega^2,\\
			\ul{\theta}_s&=\theta_s=\frac{2}{s},\\
			\zeta_s&=0.
		\end{align*}
		From this background foliation, we can explicitly compute all extrinsic curvature quantities for any spacelike cross section $\Sigma$, cf. \cite[Proposition 1]{marssoria}\footnote{Note the different sign conventions $k=-\ul{L}$, $s_l=-\zeta$.}:
		\begin{prop}\label{prop_minkowskilightcone}
			For any spacelike cross section $(\Sigma,\gamma)$ in $\mathcal{N}$, we find
			\begin{enumerate}
				\item[\emph{(i)}] $\gamma=\omega^2\d\Omega^2$,
				\item[\emph{(ii)}] $\ul{\chi}=\frac{1}{\omega}\gamma$,
				\item[\emph{(iii)}] $\ul{\theta}=\frac{2}{\omega}$,
				\item[\emph{(iv)}] $\chi=\frac{1}{\omega}(1+\btr{\nabla\omega}^2)\gamma-2\Hess\, \omega$
				\item[\emph{(v)}] $\theta=2\left(\frac{1}{\omega}+\frac{\btr{\nabla\omega}^2}{\omega}-\Delta \omega\right)$,
				\item[\emph{(vi)}] $\zeta=-\frac{\d\omega}{\omega}$,
			\end{enumerate}
			where $\Hess$ and $\Delta$ denote the Hessian and Laplace--Beltrami operator on $(\Sigma,\gamma)$, respectively.
		\end{prop}
		\begin{bem}\label{bem_minkowskilightcone}
			The fact that the null mean curvature $\ul{\chi}$ depends only pointwise on $\omega$ together with the background foliation of round spheres gives
			\[
				\newbtr{\vec{\two}}^2=\spann{\ul{\chi},\chi}=\frac{1}{2}\mathcal{H}^2,
			\]
			and thus
			\begin{align}\label{eq_gaußcurvature}
				\operatorname{R}=\frac{1}{2}\mathcal{H}^2
			\end{align}
			by the twice contracted Gauß equation Proposition \ref{prop_nullgauß}, which can also be directly verified from (iii) and (v) in Proposition \ref{prop_minkowskilightcone}. Since $\Sigma$ is $2$-dimensional, we can therefore express the Riemann tensor of the surface as
			\begin{align}\label{eq_riemannminkowski}
				\Rm_{ijkl}=\frac{1}{4}\mathcal{H}^2\left(\gamma_{ik}\gamma_{jl}-\gamma_{jk}\gamma_{il}\right).
			\end{align}
			We would like to emphasize here that $\mathcal{H}^2$ refers by definition to the signed Lorentzian length of the mean curvature tensor and can therefore be (locally) negative despite the suggestive power of $2$ as an exponent.
		\end{bem}
		In particular, we always have
		\begin{align}
			\accentset{\circ}{\vec{\two}}=-\frac{1}{2}\accentset{\circ}{\chi}\ul{L},
		\end{align}
		so $\newbtr{\accentset{\circ}{\vec{\two}}}^2=0$ although $\accentset{\circ}{\vec{\two}}\not=0$, and the property of $\vec{\two}$ being pure trace is instead more accurately captured by $\newbtr{\accentset{\circ}{\chi}}^2=0$. Along $\mathcal{N}$, this is made precise by the following lemma.
		\begin{prop}\label{prop_codazziminkowski}
			Let $(\Sigma,\gamma)$ be a spacelike cross section of $\mathcal{N}$ with $\accentset{\circ}{\chi}=0$. Then $\mathcal{H}^2$ is constant and strictly positive along $\Sigma$. In particular, $\gamma$ is a metric of constant scalar curvature.
		\end{prop}
		\begin{bem}\label{bem_codazzi}
			Note that
			\begin{align}
				\accentset{\circ}{\chi}=-2\accentset{\circ}{\Hess}_\gamma\omega=2\omega^2\accentset{\circ}{\Hess}_{\Sbb^2}\left(\frac{1}{\omega}\right),
			\end{align}
			where ${\Hess}_{\Sbb^2}$ denotes the Hessian on $(\Sbb^2,\d\Omega^2)$. One can also verify by computation in coordinates that $\accentset{\circ}{\Hess}_{\Sbb^2}\left(\frac{1}{\omega}\right)=0$ if and only if $\omega$ is of the form
			\begin{align}\label{eq_metricconstantscalar}
				\omega(\vec{x})=\frac{c}{\sqrt{1+\norm{\vec{a}}^2}+\vec{a}\cdot\vec{x}}
			\end{align}
			for $\vec{x}\in\Sbb^2$ and a fixed vector $\vec{a}\in\R^3$, which are exactly the metrics of constant scalar curvature on $\Sbb^2$. Hence, the converse statement of Proposition \ref{prop_codazziminkowski} is also true. It is a well-known fact that all such metrics can be obtained from the round metric by a suitable M\"obius transformation, cf. \cite[Proposition 6]{marssoria}, \cite[Section 5.2]{wang}. Moreover, the metrics \eqref{eq_metricconstantscalar} describe exactly the images of round spheres after a suitable Lorentz transformation in $\operatorname{SO}^+(3,1)$ in the Minkowski spacetime, which leave the lightcone $\mathcal{N}$ invariant. These observations illustrate once again the well-known fact that the M\"obius group is isomorphic to the restricted Lorentz group $\operatorname{SO}^+(3,1)$.
		\end{bem}
		\begin{proof}[Proof of Proposition \ref{prop_codazziminkowski}]
			Combining the Codazzi equation for $\chi$ from Proposition \ref{prop_nullcodazzi} with the explicit form of $\zeta$ from Proposition \ref{prop_minkowskilightcone}, we find
			\begin{align}
				\nabla_i\chi_{jk}-\frac{\d \omega_i}{\omega}\chi_{jk}=\nabla_j\chi_{ik}-\frac{\d\omega_j}{\omega}\chi_{ik}.
			\end{align}
			Multiplying the equation by $\ul{\theta}=\frac{2}{\omega}>0$, we get
			\begin{align}
				\nabla_{i}\left(\ul{\theta}\chi\right)_{jk}=\nabla_{j}\left(\ul{\theta}\chi\right)_{jk}.
			\end{align}
			Hence $\nabla\left(\ul{\theta}\chi\right)$ is totally symmetric and since $\tr_\gamma \ul{\theta}\chi=\mathcal{H}^2$, we find
			\[
				\nabla_i\mathcal{H}^2=\dive \left(\ul{\theta}\chi\right)_i=\frac{1}{2}\nabla_i\mathcal{H}^2+\dive\accentset{\circ}{\left(\ul{\theta}\chi\right)}_i=\frac{1}{2}\nabla_i\mathcal{H}^2
			\]
			by assumption. Therefore $\mathcal{H}^2$ is constant, in particular $\gamma$ is a metric of constant scalar curvature by the Gauß equation \eqref{eq_gaußcurvature}. Finally, the Gauß--Bonnet Theorem ensures the positivity of $\operatorname{R}$ and hence $\mathcal{H}^2$.
		\end{proof}
		Motivated by this, we will choose the symmetric $(0,2)$-form $A\definedas \ul{\theta}\chi$ as a scalar valued representation of the vector valued second fundamental form $\vec{\two}$. This can be regarded as a choice of gauge. Rephrasing Proposition \ref{prop_codazziminkowski} in terms of $A$, we see that we can prove the following identity in complete analogy to the properties of the scalar valued second fundamental form $h$ of an embedded, orientable surface in $\R^3$.
		\begin{prop}\label{prop_codazziminkowski2}
			Let $(\Sigma,\gamma)$ be a spacelike cross section of $\mathcal{N}$. Then $\nabla A$ is totally symmetric, i.e.,
			\begin{align}
				\nabla_iA_{jk}=\nabla_jA_{ik}.
			\end{align}
			In particular, we find
			\begin{align}\label{eq_Agradientestimate}
				\btr{\nabla A}^2\ge \frac{3}{4}\btr{\nabla\mathcal{H}^2}^2,
			\end{align}
			and $\accentset{\circ}{A}=0$ if and only if $\mathcal{H}^2$ is a strictly positive constant.
		\end{prop}
		We further derive the propagation equations for the geometric objects $A$ and $\mathcal{H}^2$ from Lemma \ref{lem_propagation} and Proposition \ref{prop_minkowskilightcone}:
		\begin{lem}\label{lem_Apropagation}
			\begin{align*}
			\frac{\d}{\d t}A_{ij}
			=&-2\Hess_{ij}(\ul{\theta}\varphi),\\
			\frac{\d}{\d t}\mathcal{H}^2
			=&-2\Delta(\ul{\theta}\varphi)-(\ul{\theta}\varphi)\mathcal{H}^2.
			\end{align*}
		\end{lem}
		\begin{proof}
			From Lemma \ref{lem_propagation} (iv) and (v), and $\ul{\chi}=\frac{1}{2}\ul{\theta}\gamma$, we compute
			\begin{align*}
				\frac{\d}{\d t}A_{ij}&=-2\ul{\theta}\Hess_{ij}\varphi-2\ul{\theta}(\d\varphi_i\otimes\zeta_j+\d\varphi_j\otimes\zeta_i)-\varphi\ul{\theta}\left(2\nabla_i\zeta_j+2\zeta_i\otimes\zeta_j-\frac{1}{2}A_{ij}\right)-\frac{1}{2}\varphi\ul{\theta}A_{ij}\\
				&=-2\ul{\theta}\Hess_{ij}\varphi-2\ul{\theta}(\d\varphi_i\otimes\zeta_j+\d\varphi_j\otimes\zeta_i)-\varphi\ul{\theta}\left(2\nabla_i\zeta_j+2\zeta_i\otimes\zeta_j\right).
			\end{align*}
			We now observe that the remaining terms on the right hand side exactly combine into $-2\Hess_{ij}\left(\ul{\theta}\varphi\right)$ due to the explicit formulas for $\ul{\theta}$ and $\zeta$ listed in Proposition \ref{prop_minkowskilightcone}. Taking a trace, where $\frac{\d}{\d t}\gamma^{ij}=-\varphi\ul{\theta}\gamma^{ij}$, completes the proof.
		\end{proof}
		We close this section by establishing a null version of the Simons' identity for $A$ in the $3+1$-Minkowski lightcone $\mathcal{N}$, which will be crucial for our later analysis.
		\begin{lem}[Null Simons' Identity]\label{lem_nullsimon}
			\begin{align*}
			\Delta A_{ij}=\Hess_{ij}\mathcal{H}^2+\frac{1}{2}\mathcal{H}^2\accentset{\circ}{A}_{ij}.
			\end{align*}
		\end{lem}
		\begin{proof}
			In the following lines, we will make frequent use of the Codazzi equation for A, Lemma \ref{prop_codazziminkowski2}, and the fact that for any symmetric $(0,2)$ tensor $T$, we have that
			\[
			\nabla_k\nabla_lT_{ij}-\nabla_l\nabla_kT_{ij}=\Rm_{kljm}T^m_i+\Rm_{klim}T^m_j.
			\]
			Thus, we compute
			\begin{align*}
			\nabla_k\nabla_lA_{ij}
			=&\nabla_k\left(\nabla_iA_{lj}\right)\\
			=&\nabla_i\nabla_k A_{jl}+\Rm_{kilm}A^m_j+\Rm_{kijm}A^m_l\\
			=&\nabla_i(\nabla_j A_{kl})+\Rm_{kilm}A^m_j+\Rm_{kijm}A^m_l\\
			=&\nabla_i\nabla_j A_{kl}+\frac{1}{4}\mathcal{H}^2\left(\left(\gamma_{kl}\gamma_{im}-\gamma_{il}\gamma_{km}\right)A^m_j+\left(\gamma_{kj}\gamma_{im}-\gamma_{ij}\gamma_{km}\right)A^m_l\right)\\
			=&\nabla_i\nabla_j A_{kl}+\frac{1}{4}\mathcal{H}^2\left(A_{ij}\gamma_{kl}+A_{il}\gamma_{kj}-A_{kj}\gamma_{il}-A_{kl}\gamma_{ij}\right),
			\end{align*}
			where we have used \eqref{eq_riemannminkowski} in the second to last line. Taking a trace with respect to the $kl$ entries yields the claim.
		\end{proof}
	\section{$2d$-Ricci flow along the standard Minkowski lightcone}\label{sec_equialence}
	We will now investigate null mean curvature flow restricted to the (past-pointing) standard lightcone in the $3+1$-Minkowski spacetime. Recall that null mean curvature flow along null hypersurfaces is defined as
	\[
	\frac{\d}{\d t}x=\frac{1}{2}\spann{\vec{\mathcal{H}},\ul{L}}\ul{L}=-\frac{1}{2}\theta\ul{L},
	\]
	as first studied by Roesch--Scheuer \cite{roeschscheuer}. Note that since $\ul{L}(r)=1$, the above is equivalent to the following evolution equation for $\omega$
	\begin{align}\label{eq_nullmeancurvatureflow}
	\frac{\d}{\d t}\omega=-\frac{1}{2}\theta.
	\end{align}
	Recall that $\accentset{\circ}{\Ric}=0$ in dimension $2$, so $2d$-Ricci flow agrees with the Yamabe flow \cite{hamilton2} and naturally preserves the conformal class of the metric. More explicitly, in our case
	\begin{align*}
	\frac{\d}{\d t}\gamma_{ij}=-2\Ric_{ij}&\Leftrightarrow \frac{\d}{\d t}\left(\omega^2\d\Omega^2\right)=-2\mathcal{K}\omega^2\d\Omega^2\\
	&\Leftrightarrow \frac{2}{\omega}\left(\frac{\d}{\d t}\omega\right)\gamma_{ij}=-2\mathcal{K}\gamma_{ij}\\
	&\Leftrightarrow \frac{\d}{\d t}\omega=-\omega\mathcal{K},
	\end{align*}
	where $\mathcal{K}$ denotes the Gauß curvature. Note that by the twice contracted Gauß equation \eqref{eq_gaußcurvature} and the explicit form of $\ul{\theta}$, we have that $\theta=2\omega\mathcal{K}$. Therefore, $2$-dimensional Ricci flow in the conformal class of the round sphere is equivalent to null mean curvature flow on the past-pointing standard lightcone in the $3+1$-Minkowski spacetime. Since $2d$-Ricci flow and its renormalized equation are fully understood in this case, we find the following corollary as a consequence of the Gauß equation \eqref{eq_gaußcurvature}, Remark \ref{bem_codazzi}, and a classical result of $2d$-Ricci flow first proven by Hamilton \cite{hamilton1}, where the initial restriction to metrics of strictly positive scalar curvature in the case of surfaces of genus $0$ was later removed by Chow \cite{chow1}:
	\begin{thm}[Hamilton and Chow, {\cite[Corollary 1.3]{chow1}}]
		If $g$ is any metric on a Riemann surface, then under Hamilton's Ricci flow, $g$ converges to a metric of constant curvature.
	\end{thm}
	\begin{kor}\label{thm_singularities}
		Let $(\Sigma_0,\gamma_0)$ be a spacelike cross section of the past-pointing standard lightcone $\mathcal{N}$ in the $3+1$-Minkowski spacetime. Then the solution of null mean curvature flow starting from $\Sigma_0$ extinguishes in the tip of the cone in finite time and the renormalization by volume converges to a surface of constant spacetime mean curvature, which exactly arise as the image of a round sphere of a Lorentz transformation in $\operatorname{SO}^+(3,1)$ consisting of a Lorentz boost with boost vector
		\[
		\vec{v}=\begin{pmatrix}
		\sqrt{1+\norm{\vec{a}}}\\\vec{a}
		\end{pmatrix}
		\]
		for a vector $\vec{a}\in \R^3$ and a rotation determined by the choice of coordinates on $\Sbb^2$.
	\end{kor}
	Conversely, we will show in the next section that by studying null mean curvature flow along the standard Minkowski lightcone a new proof for renormalized $2d$-Ricci flow arises.
	\section{A new proof of Hamilton's Theorem}\label{sec_hamilton}
	With this approach to $2d$-Ricci flow, we give a new proof of Hamilton's classical result:
	\begin{thm}[cf. \cite{hamilton1}]\label{thm_mainthm}
		Let $(\Sigma_0,\gamma_0)$ be a surface with conformally round metric $\gamma_0$ and strictly positive scalar curvature. Then a solution of renormalized Ricci flow exists for all time and the metrics $\gamma_t$ converge to a smooth metric $\gamma_\infty$ of constant scalar curvature in $C^k$ for all $k\in\N$ as $t\to\infty$.
	\end{thm}
	Note that the assumption of strictly positive scalar curvature translates by the Gauß equation \eqref{eq_gaußcurvature} to the assumption that the mean curvature vector $\vec{\mathcal{H}}$ is everywhere spacelike. Throughout this section, we will use the extrinsic objects $A$, $\mathcal{H}^2$ evolving under null mean curvature flow on the standard lightcone in the $3+1$-Minkowski spacetime, but will frequently exploit its equivalence to $2d$-Ricci flow to switch freely between the frameworks of null mean curvature flow and $2d$-Ricci flow. A key tool in the proof will be to first study the evolution of $\newbtr{\accentset{\circ}{A}}^2$ along the unnormalized flow which can be combined into the evolution of $\mathcal{H}^2$ to yield a crucial gradient estimate. Translating these to the renormalized flow will then yield the proof of Theorem \ref{thm_mainthm}.
	\begin{bem}
		Note that there does not seem to be a direct connection between $\accentset{\circ}{A}$ and the auxiliary term $M=\accentset{\circ}{\Hess}f$ in the modified renormalized flow in \cite{bartzstruweye, hamilton1}, where $f$ solves
		\[
			\Delta f=\left(\operatorname{R}-\fint_\Sigma \operatorname{R}\right).
		\]
		To see this, consider any stationary point of the renormalized flow where $f$ is necessarily constant while $\accentset{\circ}{A}=0$ holds for all functions $\omega$ of the form \eqref{eq_metricconstantscalar} arising from Lorentz transformations.
	\end{bem}
	We start by computing the relevant evolution equations for the unnormalized flow.
	\begin{prop}\label{prop_nullmeancurvature1} For a smooth solution to null mean curvature flow, we find
		\begin{align*}
		\frac{\d}{\d t}\btr{A}^2&=\Delta\btr{A}^2-2\btr{\nabla A}^2+\frac{1}{2}\left(\mathcal{H}^2\right)^3,\\
		\frac{\d}{\d t}\mathcal{H}^2&=\Delta\mathcal{H}^2+\frac{1}{2}\left(\mathcal{H}^2\right)^2.
		\end{align*}
	\end{prop}
	\begin{proof}
		For $\varphi=-\frac{1}{2}\theta$, the evolution equation for $\mathcal{H}^2$ is immediate form Lemma \ref{lem_Apropagation}. Combining the evolution equation for A from Lemma \ref{lem_Apropagation} with the null Simons' identity, Lemma \ref{lem_nullsimon}, we have
		\[
		\frac{\d}{\d t}A_{ij}=\Delta A_{ij}-\frac{1}{2}\mathcal{H}^2\accentset{\circ}{A}_{ij}.
		\]
		Hence
		\begin{align*}
		\frac{\d}{\d t}\btr{A}^2
		&=\frac{\d}{\d t}\left(\gamma^{ik}\gamma^{jl}A_{ij}A_{jk}\right)\\
		&=2\gamma^{ik}\gamma^{jl}A_{ij}\frac{\d}{\d t}A_{kl}+2\frac{\d}{\d t}\gamma^{ik}\gamma^{jl}A_{ij}A_{kl}\\
		&=2\spann{A,\Delta A-\frac{1}{2}\mathcal{H}^2\accentset{\circ}{A}}+\mathcal{H}^2\btr{A}^2\\
		&=\Delta\btr{A}^2-2\btr{\nabla A}^2-\mathcal{H}^2\newbtr{\accentset{\circ}{A}}^2+\mathcal{H}^2\btr{A}^2\\
		&=\Delta\btr{A}^2-2\btr{\nabla A}^2+\frac{1}{2}\left(\mathcal{H}^2\right)^3.
		\end{align*}
	\end{proof}
	Therefore, as we already know from Ricci flow, cf. \cite[Corollary 2.11]{chowluni}, the positivity of $\mathcal{H}^2$ is preserved under the flow by the parabolic maximum principle \cite[Proposition 2.9]{chowluni}. In particular, the flow exists only for finite time, as $\mathcal{H}^2\to\infty$ in finite time. 
	\begin{prop}\label{prop_nullmeancurvature2}
		\begin{align*}
		\frac{\d}{\d t}\frac{\btr{A}^2}{\left(\mathcal{H}^2\right)^2}
		&=\Delta\left(\frac{\btr{A}^2}{\left(\mathcal{H}^2\right)^2}\right)+\frac{2}{\mathcal{H}^2}\spann{\nabla \mathcal{H}^2,\nabla\left(\frac{\btr{A}^2}{\left(\mathcal{H}^2\right)^2}\right)}-\mathcal{H}^2\frac{\newbtr{\accentset{\circ}{A}}^2}{\left(\mathcal{H}^2\right)^2}-2\btr{\nabla\frac{A}{\mathcal{H}^2}}^2.
		\end{align*}
	\end{prop}
	\begin{bem}\label{bem_blowup}
		In particular, any upper bound on $\frac{\btr{A}^2}{\left(\mathcal{H}^2\right)^2}$ is preserved under the flow, so for a sequence $(p,t)$ of points and times, we have that $\btr{A}^2\to\infty$ if and only if $\mathcal{H}^2\to\infty$.
	\end{bem}
	\begin{proof}
		By Proposition \ref{prop_nullmeancurvature1}, we have that
		\begin{align*}
		\frac{\d}{\d t}\frac{\btr{A}^2}{\left(\mathcal{H}^2\right)^2}
		&=\frac{\frac{\d}{\d t}\btr{A}^2}{\left(\mathcal{H}^2\right)^2}-\frac{2}{\left(\mathcal{H}^2\right)^3}\btr{A}^2\frac{\d}{\d t}\mathcal{H}^2\\
		&=\frac{\Delta\btr{A}^2}{\left(\mathcal{H}^2\right)^2}-\frac{2\btr{\nabla A}^2}{\left(\mathcal{H}^2\right)^2}+\frac{1}{2}\mathcal{H}^2-\frac{2}{\left(\mathcal{H}^2\right)^3}\btr{A}^2\Delta\mathcal{H}^2-\frac{\btr{A}^2}{\left(\mathcal{H}^2\right)^2}\mathcal{H}^2\\
		&=\frac{\Delta\btr{A}^2}{\left(\mathcal{H}^2\right)^2}-\frac{2\btr{\nabla A}^2}{\left(\mathcal{H}^2\right)^2}-\frac{2}{\left(\mathcal{H}^2\right)^3}\btr{A}^2\Delta\mathcal{H}^2-\mathcal{H}^2\frac{\newbtr{\accentset{\circ}{A}}^2}{\left(\mathcal{H}^2\right)^2}.
		\end{align*}
		Note that
		\begin{align*}
		\Delta\left(\frac{\btr{A}^2}{\left(\mathcal{H}^2\right)^2}\right)
		&=\frac{\Delta \btr{A}^2}{\left(\mathcal{H}^2\right)^2}-\frac{4}{\left(\mathcal{H}^2\right)^3}\spann{\nabla\btr{A}^2,\nabla \mathcal{H}^2}+\frac{6}{\left(\mathcal{H}^2\right)^4}\btr{A}^2\btr{\nabla \mathcal{H}^2}^2-\frac{2}{\left(\mathcal{H}^2\right)^3}\btr{A}^2\Delta\mathcal{H}^2.
		\end{align*}
		Thus
		\begin{align*}
		\frac{\d}{\d t}\frac{\btr{A}^2}{\left(\mathcal{H}^2\right)^2}
		&=\Delta\left(\frac{\btr{A}^2}{\left(\mathcal{H}^2\right)^2}\right)+\frac{4}{\left(\mathcal{H}^2\right)^3}\spann{\nabla\btr{A}^2,\nabla \mathcal{H}^2}-\frac{6}{\left(\mathcal{H}^2\right)^4}\btr{A}^2\btr{\nabla \mathcal{H}^2}^2-\frac{2\btr{\nabla A}^2}{\left(\mathcal{H}^2\right)^2}-\mathcal{H}^2\frac{\newbtr{\accentset{\circ}{A}}^2}{\left(\mathcal{H}^2\right)^2}.
		\end{align*}
		Moreover, 
		\begin{align*}
		\spann{\nabla\left(\frac{\btr{A}^2}{\left(\mathcal{H}^2\right)^2}\right),\nabla\mathcal{H}^2}=\frac{1}{\left(\mathcal{H}^2\right)^2}\spann{\nabla\btr{A}^2,\nabla\mathcal{H}^2}-\frac{2}{\left(\mathcal{H}^2\right)^3}\btr{A}^2\btr{\nabla\mathcal{H}^2}^2.
		\end{align*}
		Therefore, we conclude that
		\begin{align*}
		\frac{\d}{\d t}\frac{\btr{A}^2}{\left(\mathcal{H}^2\right)^2}
		=&\,\Delta\left(\frac{\btr{A}^2}{\left(\mathcal{H}^2\right)^2}\right)+\frac{2}{\mathcal{H}^2}\spann{\nabla\left(\frac{\btr{A}^2}{\left(\mathcal{H}^2\right)^2}\right),\nabla\mathcal{H}^2} -\mathcal{H}^2\frac{\newbtr{\accentset{\circ}{A}}^2}{\left(\mathcal{H}^2\right)^2}\\
		&-\frac{2}{\left(\mathcal{H}^2\right)^4}\btr{A}^2\btr{\nabla \mathcal{H}^2}^2-\frac{2\btr{\nabla A}^2}{\left(\mathcal{H}^2\right)^2}+\frac{2}{\left(\mathcal{H}^2\right)^3}\spann{\nabla\btr{A}^2,\nabla\mathcal{H}^2}.
		\end{align*}
		Now notice that
		\begin{align*}
		\nabla_i\frac{A_{jk}}{\mathcal{H}^2}&=\frac{1}{\mathcal{H}^2}\nabla_iA_{jk}-\frac{1}{\left(\mathcal{H}^2\right)^2}A_{jk}\nabla_i\mathcal{H}^2\\
		&=\frac{1}{\left(\mathcal{H}^2\right)^2}(\mathcal{H}^2\nabla_iA_{jk}-A_{jk}\nabla_i\mathcal{H}^2),
		\end{align*}
		so that
		\begin{align*}
		\btr{\nabla\frac{A}{\mathcal{H}^2}}
		&=\frac{1}{\left(\mathcal{H}^2\right)^4}\left(\left(\mathcal{H}^2\right)^2\btr{\nabla A}^2+\btr{A}^2\btr{\nabla \mathcal{H}^2}^2-2\mathcal{H}^2A^{jk}\nabla_iA_{jk}\nabla_i\mathcal{H}^2\right)\\
		&=\frac{1}{\left(\mathcal{H}^2\right)^4}\left(\left(\mathcal{H}^2\right)^2\btr{\nabla A}^2+\btr{A}^2\btr{\nabla \mathcal{H}^2}^2-\mathcal{H}^2\spann{\nabla\btr{A}^2,\nabla\mathcal{H}^2}\right).
		\end{align*}
		The claim follows from using the last identity in the evolution equation.
	\end{proof}
	We would like to point out the similarity to the evolution of the corresponding quantities for mean curvature flow in Euclidean space, where the second fundamental form $h$ and mean curvature $H$ are replaced here by $A$ and $\mathcal{H}^2$ in the evolution equations. However, the presence of the additional good term $-\mathcal{H}^2\frac{\newbtr{\accentset{\circ}{A}}^2}{\left(\mathcal{H}^2\right)^2}$ will allow us to prove a scale breaking estimate similar to work of Huisken (cf. \cite[Theorem 8.6]{andrewschowguentherlangford}) for mean curvature flow without any pinching condition and without the need to employ a Stampacchia iteration.
	\begin{thm}\label{thm_nullmeancurvature1}
		Let $\{\Sigma_t\}_{t\in[0,T_{max})}$ be a family of closed, topological $2$-spheres with strictly positive scalar curvature evolving under Ricci flow in the form of null mean curvature flow. Then, for any $\sigma\in[0,1]$, there exists $C=C(\Sigma_0,\sigma)$, such that
		\[
		\frac{\newbtr{\accentset{\circ}{A}}^2}{\left(\mathcal{H}^2\right)^2}\le C(\mathcal{H}^2)^{-\sigma}
		\]
		for all $0\le t<T_{max}$.
	\end{thm}
	\begin{proof}
		We define $f_0\definedas \frac{\newbtr{\accentset{\circ}{A}}^2}{\left(\mathcal{H}^2\right)^2}$, and hence $f_0=\frac{\btr{A}^2}{\left(\mathcal{H}^2\right)^2}-\frac{1}{2}$ as $\tr_\gamma A=\mathcal{H}^2$. By Proposition \ref{prop_nullmeancurvature2} the evolution of $f_0$ is given by
		\[
		\frac{\d}{\d t}f_0=\Delta f_0+\frac{2}{\mathcal{H}^2}\spann{\nabla\mathcal{H}^2,\nabla f_0}-\mathcal{H}^2f_0-2\btr{\nabla\frac{A}{\mathcal{H}^2}}^2.
		\]
		We now consider $f_\sigma\definedas (\mathcal{H}^2)^\sigma f_0$ for some $\sigma\in[0,1]$. Then
		\begin{align*}
		\frac{\d}{\d t}f_\sigma
		=&\,(\mathcal{H}^2)^\sigma\frac{\d}{\d t}f_0+\sigma(\mathcal{H}^2)^{\sigma-1}f_0\frac{\d}{\d t}\mathcal{H}^2\\
		=&\,(\mathcal{H}^2)^\sigma\Delta f_0+\sigma(\mathcal{H}^2)^{\sigma-1}f_0\Delta \mathcal{H}^2+2(\mathcal{H}^2)^{\sigma-1}\spann{\nabla f_0,\nabla \mathcal{H}^2}\\
		&-\left(\mathcal{H}^2\right)^{\sigma+1}\left(1-\frac{1}{2}\sigma\right)f_0-2\left(\mathcal{H}^2\right)^{\sigma}\btr{\nabla\frac{A}{\mathcal{H}^2}}.
		\end{align*}
		Similar to before, we compute that
		\begin{align*}
		\Delta f_\sigma
		=&\,(\mathcal{H}^2)^\sigma \Delta f_0+2\sigma(\mathcal{H}^2)^{\sigma-1}\spann{\nabla f_0,\nabla \mathcal{H}^2}\\
		&+\sigma(\sigma-1)f_0(\mathcal{H}^2)^{\sigma-2}\btr{\nabla\mathcal{H}^2}^2+\sigma(\mathcal{H}^2)^{\sigma-1}f_0\Delta\mathcal{H}^2,
		\end{align*}
		and
		\begin{align*}
		\spann{\nabla f_\sigma,\nabla\mathcal{H}^2}&=(\mathcal{H}^2)^\sigma\spann{\nabla f_0,\nabla\mathcal{H}^2}+\sigma(\mathcal{H}^2)^{\sigma-1}f_0\btr{\nabla\mathcal{H}^2}^2.
		\end{align*}
		Therefore, we have that
		\begin{align*}
		\frac{\d}{\d t}f_\sigma
		=&\,\Delta f_\sigma+2(1-\sigma)(\mathcal{H}^2)^{\sigma-1}\spann{\nabla f_0,\nabla\mathcal{H}^2}+\sigma(1-\sigma)f_0(\mathcal{H}^2)^{\sigma-2}\btr{\nabla\mathcal{H}^2}^2\\
		&-\left(\mathcal{H}^2\right)^{\sigma+1}\left(1-\frac{1}{2}\sigma\right)f_0-2\left(\mathcal{H}^2\right)^{\sigma}\btr{\nabla\frac{A}{\mathcal{H}^2}}\\
		=&\,\Delta f_\sigma+\frac{2(1-\sigma)}{\mathcal{H}^2}\spann{\nabla f_\sigma,\nabla \mathcal{H}^2}-\sigma(1-\sigma)f_\sigma(\mathcal{H}^2)^{-2}\btr{\nabla\mathcal{H}^2}^2\\
		&-\mathcal{H}^2\left(1-\frac{1}{2}\sigma\right)f_\sigma-2\left(\mathcal{H}^2\right)^{\sigma}\btr{\nabla\frac{A}{\mathcal{H}^2}}\\
		\le&\,\Delta f_\sigma+\frac{2(1-\sigma)}{\mathcal{H}^2}\spann{\nabla f_\sigma,\nabla \mathcal{H}^2}
		\end{align*}
		for $\sigma\in[0,1]$. Thus the claim follows by the parabolic maximum principle.
	\end{proof}
	Using this estimate, we establish a gradient bound, that will allow us to conclude that the ratio between the minimum and maximum of $\mathcal{H}^2$ and therefore $\operatorname{R}$ converges to $1$ as $t\to T_{max}$. We will state this in the framework of Ricci flow.
	\begin{thm}\label{thm_gradientestiamte}
		Let $\{\Sigma_t\}_{t\in[0,T_{max})}$ be a family of closed, topological $2$-spheres with strictly positive scalar curvature evolving under Ricci flow. For any $\eta>0$, there exists $C_{\eta}>0$ only depending on $\eta$ and $\Sigma_0$, such that
		\begin{align*}
		\btr{\nabla \operatorname{R}}\le \eta^2\operatorname{R}^{\frac{3}{2}}+C_\eta.
		\end{align*}
		for all $t\in[0,T_{max})$.
	\end{thm}
	\begin{bem}\label{bem_gradientestimate}
		As $\mathcal{H}^2=2\operatorname{R}$ by the Gauß Equation \eqref{eq_gaußcurvature}, it suffices to proof
		\[
			\btr{\nabla \mathcal{H}^2}\le \eta^2\left(\mathcal{H}^2\right)^{\frac{3}{2}}+C_\eta
		\]
		From this, we get the crucial gradient estimate 
		\[
		\btr{\nabla \operatorname{R}}\le \eta^2 \operatorname{R}^{\frac{3}{2}}_{max}
		\]
		for any $\eta>0$ and $t$ sufficiently close to $T_{max}$. This estimate allowed Hamilton to conclude that the ratio of $\operatorname{R}_{min}$ ad $\operatorname{R}_{max}$ converges to $1$ in the $3$-dimensional case using the Theorem of Bonnet--Myers, cf. \cite[Lemma 3.22]{chowluni}. We now established this estimate for $2$-dimensional Ricci flow and can argue analogously. Compare also the corresponding result by Huisken for mean curvature flow of convex surfaces, cf. \cite[Corollary 8.16]{andrewschowguentherlangford}.
	\end{bem}
	\begin{kor}\label{kor_gradientestimate}
		As $t\to T_{max}$,
		\begin{align*}
		\frac{\mathcal{H}^2_{max}}{\mathcal{H}^2_{min}}=\frac{\operatorname{R}_{max}}{\operatorname{R}_{min}}&\to 1,\\
		\operatorname{diam}(\Sigma_t)&\to 0.
		\end{align*}
	\end{kor}
	\begin{proof}[Proof of Theorem \ref{thm_gradientestiamte}]
		We fix any $\sigma\in(0,1]$, and by Theorem \ref{thm_nullmeancurvature1}, there exists $C$ only depending on $\Sigma_0$ (and our fixed choice of $\sigma$), such that 
		\begin{align}\label{eq_thmgradientestimate1}
		\frac{\newbtr{\accentset{\circ}{A}}^2}{\left(\mathcal{H}^2\right)^2}\le C(\mathcal{H}^2)^{-\sigma}.
		\end{align}
		Using Young's Inequality, we find that for any $\varepsilon>0$, there exists $C_\varepsilon>0$ (only depending on $\varepsilon$ and $\Sigma_0$), such that
		\begin{align}\label{eq_thmgradietestiamte2}
		\newbtr{\accentset{\circ}{A}}^2\le \varepsilon\left(\mathcal{H}^2\right)^2+C_\varepsilon.
		\end{align}
		We define
		\begin{align}\label{eq_thmgradietestiamte3}
		G_\varepsilon\definedas 2C_\varepsilon+(\varepsilon+\frac{1}{2})\left(\mathcal{H}^2\right)^2-\btr{A}^2=2 C_\varepsilon+\varepsilon\left(\mathcal{H}^2\right)^2-\newbtr{\accentset{\circ}{A}}^2\ge C_\varepsilon>0.
		\end{align}
		Computing the evolution of \eqref{eq_thmgradietestiamte3} and abbreviating $\frac{\d}{\d t}$ as $\partial_t$, we find that
		\begin{align*}
		\left(\partial_t-\Delta\right)G_\varepsilon
		&=\varepsilon\left(\mathcal{H}^2\right)^3+2\btr{\nabla A}^2-(1+2\varepsilon)\btr{\nabla\mathcal{H}^2}^2.
		\end{align*}
		Recall that $\btr{\nabla A}^2\ge \frac{3}{4}\btr{\nabla\mathcal{H}^2}^2$ as proven in Proposition \ref{prop_codazziminkowski2}. Hence, for any $0<\varepsilon\le \frac{1}{8}$ we find
		\[
		\btr{\nabla A}^2-(\frac{1}{2}+\varepsilon)\btr{\nabla\mathcal{H}^2}\ge \frac{1}{8}\btr{\nabla\mathcal{H}^2}^2,
		\]
		and thus we have
		\begin{align}\label{eq_thmgradientestimate4}
		\left(\partial_t-\Delta\right)G_\varepsilon\ge \frac{1}{4}\btr{\nabla \mathcal{H}^2}^2
		\end{align}
		for any $0<\varepsilon\le\frac{1}{8}$.
		Also recall that by Proposition \ref{prop_nullmeancurvature1}, we have
		\begin{align}\label{eq_thmgradientestimate5}
		\left(\partial_t-\Delta\right)\mathcal{H}^2\ge 0,
		\end{align}
		and since $\mathcal{H}^2=2R$ and we are evolving the surface under Ricci flow, we have that
		\begin{align}\label{eq_thmgradientestimate6}
		\left(\partial_t-\Delta\right)\btr{\nabla \mathcal{H}^2}^2&\le C\mathcal{H}^2\btr{\nabla \mathcal{H}^2}^2-2\btr{\nabla^2\mathcal{H}^2}^2,
		\end{align}
		cf. \cite[Chapter 3]{chowluni}. We now look at a new maximum of $\frac{\btr{\nabla\mathcal{H}^2}^2}{G_\varepsilon\mathcal{H}^2}$ at a time and point $(t,p)$. In particular, we have that
		\[
		0=\nabla\frac{\btr{\nabla\mathcal{H}^2}^2}{G_\varepsilon\mathcal{H}^2}=2\frac{\nabla^i\mathcal{H}^2\nabla\nabla_i\mathcal{H}^2}{G_\varepsilon\mathcal{H}^2}-\frac{\btr{\nabla\mathcal{H}^2}^2}{G_\varepsilon\mathcal{H}^2}\left(\frac{\nabla G_\varepsilon}{G_\varepsilon}+\frac{\nabla \mathcal{H}^2}{\mathcal{H}^2}\right),
		\]
		so at $(p,t)$ it holds that
		\begin{align}\label{eq_thmgradientestimate7}
		4\frac{\btr{\nabla\mathcal{H}^2}^2}{G_\varepsilon\mathcal{H}^2}\spann{\frac{\nabla G_\varepsilon}{G_\varepsilon},\frac{\nabla\mathcal{H}^2}{\mathcal{H}^2}}\le \frac{\btr{\nabla\mathcal{H}^2}^2}{G_\varepsilon\mathcal{H}^2}\btr{\frac{\nabla G_\varepsilon}{G_\varepsilon}+\frac{\nabla \mathcal{H}^2}{\mathcal{H}^2}}^2\le 4\frac{\btr{\nabla^2\mathcal{H}^2}^2}{G_\varepsilon\mathcal{H}^2}.
		\end{align}
		Moreover, direct computation yields
		\begin{align*}
		\left(\partial_t-\Delta\right)\left(\frac{\btr{\nabla\mathcal{H}^2}^2}{G_\varepsilon\mathcal{H}^2}\right)
		=&\,\frac{(\partial_t-\Delta)\btr{\nabla\mathcal{H}^2}^2}{G_\varepsilon\mathcal{H}^2}-\frac{\btr{\nabla\mathcal{H}^2}^2}{G_\varepsilon\mathcal{H}^2}\left(\frac{(\partial_t-\Delta)G_\varepsilon}{G_\varepsilon}+\frac{(\partial_t-\Delta)\mathcal{H}^2}{\mathcal{H}^2}\right)\\
		&-\frac{2}{G_\varepsilon\mathcal{H}^2}\spann{\nabla\left(\frac{\btr{\nabla\mathcal{H}^2}^2}{G_\varepsilon\mathcal{H}^2}\right),\nabla(G_\varepsilon\mathcal{H}^2)}-\frac{2\btr{\nabla\mathcal{H}^2}^2}{G_\varepsilon\mathcal{H}^2}\spann{\frac{\nabla G_\varepsilon}{G_\varepsilon},\frac{\nabla\mathcal{H}^2}{\mathcal{H}^2}}.
		\end{align*}
		Hence, using \eqref{eq_thmgradientestimate4}, \eqref{eq_thmgradientestimate5}, \eqref{eq_thmgradientestimate6}, and \eqref{eq_thmgradientestimate7}, we see that at $(p,t)$
		\begin{align*}
		0&\le \frac{1}{G_\varepsilon\mathcal{H}^2}\left(C(n)\mathcal{H}^2\btr{\nabla\mathcal{H}^2}^2-2\btr{\nabla^2\mathcal{H}^2}^2\right)-\frac{\btr{\nabla\mathcal{H}^2}^2}{G_\varepsilon\mathcal{H}^2}\left(\frac{1}{4}\frac{\btr{\nabla\mathcal{H}^2}^2}{G_\varepsilon}\right)+\frac{2\btr{\nabla^2\mathcal{H}^2}^2}{G_\varepsilon\mathcal{H}^2}\\
		&=\frac{\btr{\nabla\mathcal{H}^2}^2}{G_\varepsilon\mathcal{H}^2}\left(C(n)\mathcal{H}^2-\frac{1}{4}\frac{\btr{\nabla\mathcal{H}^2}^2}{G_\varepsilon}\right),
		\end{align*}
		so after rearranging, we have that
		\begin{align*}
		\frac{\btr{\nabla\mathcal{H}^2}^2}{G_\varepsilon\mathcal{H}^2}\le 4C(n)
		\end{align*}
		at any new maximum. So we find that
		\begin{align*}
		\frac{\btr{\nabla\mathcal{H}^2}^2}{G_\varepsilon\mathcal{H}^2}\le \max\left(\max\limits_{\Sigma_0}\frac{\btr{\nabla\mathcal{H}^2}^2}{G_\varepsilon\mathcal{H}^2}, 4C(n)\right)=:\widetilde{C},
		\end{align*}
		and in particular
		\begin{align*}
		\btr{\nabla\mathcal{H}^2}^2\le\widetilde{C}G_\varepsilon\mathcal{H}^2\le \widetilde{C}\mathcal{H}^2\left(\varepsilon\left(\mathcal{H}^2\right)^2+2C_\varepsilon\right).
		\end{align*}
		After taking a square root, the proof now follows for any $\eta>0$ using Young's inequality again and choosing $0<\varepsilon<\frac{1}{8}$ sufficiently small.
	\end{proof}
	We briefly recall some properties of $n$-dimensional Ricci flow renormalized by volume (cf. \cite[Chapter 3.6]{chowluni}), i.e.,
	\begin{align}\label{eq_renormricciflow}
	\frac{\d} {\d\widetilde{t}}\widetilde{\gamma}(\,\widetilde{t}\,)&=-\widetilde{\Ric}(\,\widetilde{t}\,)+\frac{2}{n}\widetilde{r}\widetilde{\gamma}(\,\widetilde{t}\,),
	\end{align}
	where $\widetilde{r}=\fint\widetilde{\operatorname{R}}$, such that along any solution we have that $\operatorname{Vol}(\widetilde{\gamma}(\,\widetilde{t}\,))=\operatorname{Vol}(\widetilde{\gamma}(0))=\operatorname{const.}$. Given a solution of Ricci flow $\gamma(t)$, $t\in[0,T)$, the metrics $\widetilde{\gamma}(\,\widetilde{t}\,)\definedas c(t)\gamma(t)$, with
	\[
	c(t)\definedas\exp\left(\frac{2}{n}\int\limits_0^tr(\tau)\d\tau\right),\text{ }\widetilde{t}(t)\definedas\int\limits_0^tc(\tau)\d\tau,
	\]
	satisfy \eqref{eq_renormricciflow} with initial condition $\widetilde{\gamma}(0)=\gamma(0)$, so we can always renormalize a given solution of Ricci flow. Moreover, we have the following transformation laws for evolution equations by Hamilton:
	\begin{lem}[{Hamilton, see \cite[Lemma 3.26]{chowluni}}]\label{lemma_renormalizedflow}
		If an expression $X=X(\gamma)$ formed algebraically from the metric and the Riemann curvature tensor has degree $k$, i.e.,\linebreak $X(c\gamma)=c^kX(\gamma)$, and if under the Ricci flow
		\[
		\partial_tX=\Delta X+Y,
		\]
		then the degree of $Y$ is $k-1$ and the evolution under the normalized Ricci flow \linebreak
		{$\frac{\partial} {\partial\widetilde{t}}\widetilde{\gamma}_{ij}=-\widetilde{\Ric}_{ij}+\frac{2}{n}\widetilde{r}\widetilde{\gamma}_{ij}$}
		of $\widetilde{X}:=X(\widetilde{\gamma})$ is given by
		\[
		\partial_{\widetilde{t}}\widetilde{X}=\widetilde{\Delta}\widetilde{X}+\widetilde{Y}+k\frac{2}{n}\widetilde{X}.
		\]
	\end{lem}
	Recall that by \cite[Remark 3.27]{chowluni}, this lemma also extends to the corresponding partial differential inequalities if $Y$ is of degree $k-1$, and is further also applicable to arbitrary tensor derivatives of such expressions as used frequently throughout \cite[Chapter 3]{chowluni}.
	\begin{bem}
		In this section, as before, we will only look at the case when $n=2$ and $\gamma(t)$ is conformal to the round sphere for each $t$. Thus $\widetilde{\gamma}(\,\widetilde{t}\,)$ is conformally round, and by the Gauß--Bonnet Theorem
		\[	
		\widetilde{r}(t)=\frac{8\pi}{\operatorname{Vol}(\widetilde{\gamma}(\,\widetilde{t}\,))}=\frac{8\pi}{\operatorname{Vol}(\widetilde{\gamma}(0))}
		\]
		is positive and remains constant along the flow.
	\end{bem}
	From now on, will assume without loss of generality that 
	\begin{align}\label{eq_scalarcurvatureratio}
	\frac{1}{2}\operatorname{R}_{max}(t)\le \operatorname{R}(x,t)
	\end{align}
	for all $t\in[0,T_{max}), x\in\Sigma_t$ (this is ultimately satisfied for $t$ sufficiently close to $T_{max}$ due to Corollary \ref{kor_gradientestimate}). Note that due to the relation between $\mathcal{H}^2$ and $\operatorname{R}$ via the Gauß equation \eqref{eq_gaußcurvature}, combining Proposition \ref{prop_nullmeancurvature1} with the fact that $\operatorname{R}_{max}\to\infty$ as $t\to T$, we find that $\operatorname{R}_{max}\ge (T-t)^{-1}$. In particular,
	\begin{align}\label{eq_scalarcurvlowerbound}
	\operatorname{R}(t,x)\ge \frac{1}{2(T-t)}.
	\end{align}
	In the following, we will switch freely between the framework of (renormalized) $2$-d Ricci flow and null mean curvature flow along the past-pointing standard lightcone. Recall that most importantly, bounds for $\mathcal{H}^2$ and its derivatives correspond directly to bounds on the scalar curvature and its derivatives via the Gauß equation \eqref{eq_gaußcurvature}.
	In our analysis of the renormalized flow we will closely follows the outline of Hamilton's strategy presented in \cite[Chapter 3]{chowluni} for $3$-dimensional Ricci flow, and include the proofs for the sake of completeness. We start by establishing the following lemma:
	\begin{lem}\label{lem_renormalizedflow}
		For the renormalized flow \eqref{eq_renormricciflow}, we have that
		\begin{enumerate}
			\item[\emph{(i)}] $\widetilde{T}=\infty$,
			\item[\emph{(ii)}] $
			\lim\limits_{\widetilde{t}\to\infty}\frac{\widetilde{\mathcal{H}}^2_{min}}{\widetilde{\mathcal{H}}^2_{max}}=1
			$
			\item[\emph{(iii)}] There exists $C>0$ such that $\frac{1}{C}\le \widetilde{\mathcal{H}}^2_{min}(\,\widetilde{t}\,)\le \widetilde{\mathcal{H}}^2_{max}(\,\widetilde{t}\,)\le C$ for all $\widetilde{t}$,
			\item[\emph{(iv)}] $\operatorname{diam}(\widetilde{\gamma}(\,\widetilde{t}\,))\le C$
			\item [\emph{(v)}] $\newbtr{\accentset{\circ}{\widetilde{A}}}^2\le Ce^{-\delta \widetilde{t}}$,
			\item [\emph{(vi)}] $\btr{\widetilde{\nabla}\widetilde{\mathcal{H}}^2}^2\le Ce^{-\delta \widetilde{t}}$,
			\item [\emph{(vii)}] $\widetilde{\mathcal{H}}^2_{max}-\widetilde{\mathcal{H}}^2_{min}\le Ce^{-\delta \widetilde{t}}$.
		\end{enumerate}
	\end{lem}
	\begin{proof}
		In the following, $C,\widetilde{C}$ will always denote positive constants independent of $t$ that may vary from line to line.
		\begin{enumerate}
			\item[(i)] By substitution rule, we find that
			\[
			\int\limits_0^{\widetilde{t}(t_0)}\widetilde{r}(\widetilde{\tau})\d\widetilde{\tau}=\int\limits_0^{t_0}r(\tau)\d\tau.
			\]
			Combining this formula for $t\to T$ with \eqref{eq_scalarcurvlowerbound}, we can conclude (i) since $\widetilde{r}$ is constant. 
			\item[(ii)] Follows immediately from Corollary \ref{kor_gradientestimate}.
			\item[(iii)] By the Bishop--Gromov volume comparison, we have that
			\[
			\operatorname{Vol}(\widetilde{\gamma}(0))=\operatorname{Vol}(\widetilde{\gamma}(\,\widetilde{t}\,))\le C\operatorname{diam}(\widetilde{\gamma}(\,\widetilde{t}\,))^2,
			\]
			and recall that by the Bonnet--Myers Theorem
			\begin{align}\label{eq_renormalizedflow_bonnetmyers}
			\operatorname{diam}(\widetilde{\gamma}(\,\widetilde{t}\,))\le C\left(\widetilde{\mathcal{H}}_{max}^2\right)^{-\frac{1}{2}}.
			\end{align}
			Thus, $\widetilde{\mathcal{H}}_{max}^2$ is uniformly bounded from above. Now note that $\widetilde{\Sigma}_{\widetilde{t}}$ is a topological sphere, in particular simply connected. Hence, Klingenberg's injectivity radius estimate yields that
			\[
			\operatorname{inj}(\widetilde{\gamma}(\,\widetilde{t}\,))\ge C\left(\widetilde{\mathcal{H}}_{max}^2\right)^{-\frac{1}{2}},
			\]
			and therefore
			\[
			\operatorname{Vol}(\widetilde{\gamma}(0))=\operatorname{Vol}(\widetilde{\gamma}(\,\widetilde{t}\,))\ge C\left(\widetilde{\mathcal{H}}_{max}^2\right)^{-1},
			\]
			so $\widetilde{\mathcal{H}}_{max}^2$ is also uniformly bounded from below. Since the inequality  \eqref{eq_scalarcurvatureratio} is preserved under rescaling, we have that the scalar curvature is uniformly bounded and we can therefore pick some constant $C>0$ such that (iii) is satisfied.
			\item[(iv)] Follows directly from (iii) via \eqref{eq_renormalizedflow_bonnetmyers}.
		\end{enumerate}
		We prove (v) and (vi) simultaneously. We define $\Psi\definedas \frac{\btr{\nabla\mathcal{H}^2}^2}{\mathcal{H}^2}+K\newbtr{\accentset{\circ}{A}}^2$ for some positive constant $K$ to be determined later. Recall that along the unnormalized flow, we have that
		\[
		\left(\partial_t-\Delta\right)\btr{\nabla \mathcal{H}^2}^2\le C\mathcal{H}^2\btr{\nabla \mathcal{H}^2}^2-2\btr{\nabla^2\mathcal{H}^2}^2
		\]
		for some fixed constant $C$, and we find that
		\[
		\left(\partial_t-\Delta\right)\newbtr{\accentset{\circ}{A}}^2\le -\frac{1}{2}\btr{\nabla\mathcal{H}^2}^2
		\]
		by direct computation via Proposition \ref{prop_nullmeancurvature1} and $\btr{\nabla A}^2\ge \frac{3}{4}\btr{\nabla\mathcal{H}^2}^2$ proven in Proposition \ref{prop_codazziminkowski2}. Thus $\Psi$ is of degree $-2$ and satisfies the evolution equation
		\begin{align*}
		\partial_t\Psi&\le \frac{\Delta\btr{\nabla\mathcal{H}^2}^2}{\mathcal{H}^2}-\frac{\btr{\nabla\mathcal{H}^2}^2}{\left(\mathcal{H}^2\right)^2}\Delta\mathcal{H}+K\Delta \newbtr{\accentset{\circ}{A}}^2-2\frac{\btr{\nabla^2\mathcal{H}^2}^2}{\mathcal{H}^2}+\left(C-\frac{1}{2}-\frac{K}{2}\right)\btr{\nabla\mathcal{H}^2}^2\\
		&=\Delta\Psi+2\frac{\spann{\nabla\btr{\nabla\mathcal{H}^2}^2,\nabla\mathcal{H}^2}}{\left(\mathcal{H}^2\right)^2}-2\frac{\btr{\nabla^2\mathcal{H}^2}^2}{\mathcal{H}^2}-2\frac{\btr{\nabla\mathcal{H}^2}^4}{\left(\mathcal{H}^2\right)^3}+\left(C-\frac{1}{2}-\frac{K}{2}\right)\btr{\nabla\mathcal{H}^2}^2.
		\end{align*}
		Note that
		\begin{align*}
		\frac{\spann{\nabla\btr{\nabla\mathcal{H}^2}^2,\nabla\mathcal{H}^2}}{\left(\mathcal{H}^2\right)^2}
		=2\frac{\nabla_k\nabla_i\mathcal{H}^2\nabla^i\mathcal{H}^2\nabla^k\mathcal{H}^2}{\left(\mathcal{H}^2\right)^2}\le 2\frac{\btr{\nabla^2\mathcal{H}^2}\btr{\nabla\mathcal{H}^2}^2}{\left(\mathcal{H}^2\right)^2}\le \frac{\btr{\nabla^2\mathcal{H}^2}^2}{\mathcal{H}^2}+\frac{\btr{\nabla\mathcal{H}^2}^2}{\left(\mathcal{H}^2\right)^3}.
		\end{align*}
		We now choose $K\definedas 2C>0$, so we can conclude that $(\partial_t-\Delta)\Psi\le 0$. In particular, $\widetilde{\Psi}$ satisfies
		\[
		(\partial_{\widetilde{t}}-\widetilde{\Delta})\widetilde{\Psi}\le -\widetilde{C}\widetilde{\Psi}.
		\]
		By the maximum principle, we can now conclude that
		\[
		\frac{\btr{\widetilde{\nabla}\widetilde{\mathcal{H}}^2}^2}{\widetilde{\mathcal{H}}^2}+K\newbtr{\accentset{\circ}{\widetilde{A}}}^2\le Ce^{-\delta\widetilde{t}},
		\]
		so (v) and (vi) follow since $\widetilde{\mathcal{H}}^2$ is uniformly bounded. Lastly (vii) follows form (iv) and (vi).
	\end{proof}
	In particular $\btr{\widetilde{\operatorname{R}}-\widetilde{r}}\le Ce^{-\widetilde{\delta} t}$, so we know that the evolution speed of the renormalized flow \eqref{eq_renormricciflow} is integrable. We thus acquire uniform bounds and $C^0$ convergence of the metric due to a Lemma by Hamilton (cf. \cite[Lemma 6.10]{chowluni}):
	\begin{kor}\label{kor_uniformbound}
		Let $\gamma(t)$ be a solution of $2d$-Ricci flow with $\operatorname{R}>0$. Then the renormalized flow \eqref{eq_renormricciflow} exists for all time, and there exists a constant $C>0$ such that
		\[
		\frac{1}{C}\widetilde{\gamma}(0)\le \widetilde{\gamma}(\widetilde{t})\le C\widetilde{\gamma}(0),
		\]
		and $\widetilde{\gamma}(\widetilde{t})$ converges uniformly to a limiting metric $\widetilde{\gamma}(\infty)$ on compact sets as $\widetilde{t}\to\infty$.
	\end{kor}
	\begin{bem}\label{bem_uniformbound}
		Since the renormalized metrics are also conformally round, i.e.,\linebreak $\widetilde{\gamma}(\,\widetilde{t}\,)=\widetilde{\omega}^2(\,\widetilde{t}\,)\d\Omega^2$, Corollary \ref{kor_uniformbound} also yields a uniform bound on the conformal factors $\widetilde{\omega}(t)$ only depending on $\widetilde\omega(0)$.
	\end{bem}
	To complete the proof of Theorem \ref{thm_mainthm}, it remains to show that $\widetilde{\gamma}(\infty)$ is in fact smooth and that the renormalized flow converges in $C^k$ for any $k$. In particular, due to Lemma \ref{lem_renormalizedflow} (i), $\widetilde{\gamma}(\infty)$ is then a metric of constant scalar curvature.
	We thus require bounds for the derivatives of the renormalized metrics, and by a standard argument it suffices to bound the derivatives of the Riemann tensor. However, for $n=2$, the Riemann tensor and its derivatives are fully determined by the scalar curvature and its derivatives. By the Gauß equation \eqref{eq_gaußcurvature}, it thus suffices to find appropriate bounds for $\widetilde{\mathcal{H}}^2$ and its derivatives.
	\begin{lem}\label{lem_renormflowhigherderivatives}
		For all $k\in\N$, there exist $C_k, \delta_k>0$, such that
		\[
		\btr{\widetilde{\nabla}^k\widetilde{\mathcal{H}}^2}^2\le C_ke^{-\delta_k\widetilde{t}}
		\]
	\end{lem}
	\begin{proof}
		Since $n=2$, there exists a fixed constant $C$, such that $\btr{\nabla^k\mathcal{H}^2}^2=C \btr{\nabla^k\Rm}^2$, and thus the evolution of $\btr{\nabla^k\mathcal{H}^2}^2$ along the unnormalized flow can be estimated by
		\begin{align}\label{eq_lem_higherderivatives1}
		\partial_t\btr{\nabla^k\mathcal{H}^2}^2\le \Delta \btr{\nabla^k\mathcal{H}^2}^2-2\btr{\nabla^{k+1}\mathcal{H}^2}^2+C(k)\sum_{l=0}^k\btr{\nabla^l\mathcal{H}^2}\btr{\nabla^{k-l}\mathcal{H}^2}\btr{\nabla^k\mathcal{H}^2},
		\end{align}
		where $C(k)$ denotes a constant only depending on $k$, cf. \cite[Chapter 3]{chowluni}.\newline
		We will proof the statement by strong induction, where in the following $C,C_k,C_{k+1}$ will be constants only depending on $k$ which may vary from line to line.\newline
		The statement is true for $k=1$ as proven in Lemma \ref{lem_renormalizedflow} (vi).\newline
		We now assume that the statement is true for all $1\le l\le k$, and proceed from $k$ to $k+1$. We define $f\definedas\btr{\nabla^{k+1}\mathcal{H}^2}^2+K\mathcal{H}^2\btr{\nabla^{k}\mathcal{H}^2}^2$, where $K$ is a positive constant to be determined later. Then $f$ is of degree $-k-3$ and according to \eqref{eq_lem_higherderivatives1} its evolution under Ricci flow is given by
		\begin{align*}
		\partial_t f
		\le &\,\Delta \btr{\nabla^{k+1}\mathcal{H}^2}^2-2\btr{\nabla^{k+2}\mathcal{H}^2}^2+C_{k+1}\sum\limits_{l=0}^{k+1}\btr{\nabla^{l}\mathcal{H}^2}\btr{\nabla^{k+1-l}\mathcal{H}^2}\btr{\nabla^{k+1}\mathcal{H}^2}\\
		&+K\btr{\nabla^{k}\mathcal{H}^2}^2\Delta\mathcal{H}^2+\frac{K}{2}\left(\mathcal{H}^2\right)^2\btr{\nabla^{k}\mathcal{H}^2}^2\\
		&+K\mathcal{H}^2\Delta \btr{\nabla^{k}\mathcal{H}^2}^2-2K\mathcal{H}^2\btr{\nabla^{k+1}\mathcal{H}^2}^2+K\mathcal{H}^2C_k\sum\limits_{l=0}^k\btr{\nabla^{l}\mathcal{H}^2}\btr{\nabla^{k-l}\mathcal{H}^2}\btr{\nabla^{k}\mathcal{H}^2}\\
		=&\,\Delta f-2K\spann{\nabla \mathcal{H}^2,\nabla\btr{\nabla^k\mathcal{H}^2}^2}-2K\mathcal{H}^2\btr{\nabla^{k+1}\mathcal{H}^2}^2\\
		&+C_{k+1}\sum\limits_{l=0}^{k+1}\btr{\nabla^{l}\mathcal{H}^2}\btr{\nabla^{k+1-l}\mathcal{H}^2}\btr{\nabla^{k+1}\mathcal{H}^2}+K\mathcal{H}^2C_k\sum\limits_{l=0}^k\btr{\nabla^{l}\mathcal{H}^2}\btr{\nabla^{k-l}\mathcal{H}^2}\btr{\nabla^{k}\mathcal{H}^2}\\
		\le\,& \Delta f+C_k(K,\mathcal{H}^2,\nabla^{1\le l \le k}\mathcal{H}^2)+(C_{k+1}-2K)\mathcal{H}^2\btr{\nabla^{k+1}\mathcal{H}^2}^2,
		\end{align*}
		where we have used Young's inequality in the last line and collected all remaining terms in $C_k(K,\mathcal{H}^2,\nabla^{1\le l \le k}\mathcal{H}^2)$.
		We now choose $K:=C_{k+1}$, and thus we find that
		\[
		(\partial_t-\Delta)f\le C_k(\mathcal{H}^2,\nabla^{1\le l \le k}\mathcal{H}^2),
		\]
		where $C_k(\mathcal{H}^2,\nabla^{1\le l \le k}\mathcal{H}^2)$ denotes a sum of products of derivatives with at least order $1$ and at most order $k$ such that the factors only depend on $k$ and possibly $\mathcal{H}^2$, and $C_k(\mathcal{H}^2,\nabla^{1\le l \le k}\mathcal{H}^2)$ is of degree $-k-4$. Hence, the evolution of $\widetilde{f}$ along the renormalized flow is given by
		\[
		(\partial_{\widetilde{t}}-\widetilde{\Delta})\widetilde{f}\le C_k(\widetilde{\mathcal{H}}^2,\widetilde{\nabla}^{1\le l\le k}\widetilde{\mathcal{H}}^2)-C\widetilde{f}\le \widetilde{C}e^{-\widetilde{\delta}\,\widetilde{t}}-C\widetilde{f}
		\]
		for some $\widetilde{C},\widetilde{\delta}>0$ by induction, as $\widetilde{\mathcal{H}}^2$ is uniformly bounded by Lemma \ref{lem_renormalizedflow}. Now choosing $\delta<\min(\widetilde{\delta},C)$, we find that
		\[
		(\partial_{\widetilde{t}}-\widetilde{\Delta})\left(e^{\delta\widetilde{t}}\widetilde{f}-\widetilde{C}\widetilde{t}\right)\le0.
		\]
		So by the maximum principle, there exists $C_0>0$ such that
		\[
		e^{\delta\widetilde{t}}\widetilde{f}-\widetilde{C}\widetilde{t}\le C_0\Leftrightarrow\widetilde{f}\le Ce^{-\delta\widetilde{t}}(C_0+\widetilde{C}\widetilde{t}).
		\]
		Since exponential decay wins over linear growth, there exists an appropriate constant $C_k>0$ for any choice $0<\delta_k<\delta$ such that
		\[
		\widetilde{f}\le C_ke^{-\delta_k\widetilde{t}}.
		\]
		This concludes the proof.
	\end{proof}
	From this, we can conclude the uniform convergence in $C^k$ for any $k\in\N$ and Theorem \ref{thm_mainthm} is proven.
\section{Comments}\label{sec_discussion}
	We close with some comments on the higher dimensional case. Since the general structure of the standard lightcone derived in \Cref{sec_nullgeom} extends directly to higher dimensions (up to some possibly dimension dependent constants), the geometric intuition developed in \Cref{sec_nullgeom} also holds for the standard lightcone in the $n+1$ dimensional Minkowski spacetime, $n\ge 3$. In particular, the Gauß equation yields
	\begin{align}\label{eq_higherdim1}
		R=\frac{n-1}{n}\mathcal{H}^2.
	\end{align}
	From this, we can similarly establish that null mean curvature flow is proportional to the Yamabe flow \cite{hamilton2} for the conformal class of the round metric on $\Sbb^{n-1}$ in all dimensions $n-1\ge 2$. More precisely, the metrics evolve under renormalized null mean curvature flow as
	\begin{align}
		\frac{\d}{\d\widetilde{t}}\widetilde{g}(\widetilde{t})=-\frac{1}{n-1}\left(\widetilde{R}-\fint\widetilde{R}\right)\widetilde{g}(\widetilde{t}).
	\end{align}
	Since not all metrics on $\Sbb^{n-1}$ are necessarily conformally round in higher dimension $n-1\ge 3$ they can thus not be embedded isometrically into the standard Minkowski lightcone.
	
	Similar to the $2$-dimensional case, renormalized Yamabe flow has been subject to thorough investigation using various different methods. The case of the conformal class of the round sphere was first treated separately by Chow \cite{chow2} under the additional assumption of positive Ricci curvature, which is preserved under the flow. A uniform approach for locally conformally flat metrics was later provided by Ye \cite{ye}. Schwetlick--Struwe \cite{schwetlickstruwe} performed a precise blow-up analysis and showed that singularities arising in the blow-up procedure can by ruled out by employing the positive mass theorem (cf. \cite{schoenyau}) if the initial energy is less than some uniform bound depending on the Yamabe invariant of the initial metric and the Yamabe energy of the round sphere. The general approach by Brendle \cite{brendle,brendle3} leads to a short proof of the conformally round case \cite{brendle2}. We suspect that the techniques developed in this paper could be applied to gain a new proof of this result, possibly under similar restrictions as Chow \cite{chow2}. This is subject of future work.
	\nocite{*}

\bibliography{bib_ricciflowlightcone}

\nopagebreak
\bibliographystyle{plain}
\end{document}